\documentclass[12pt]{article}
\usepackage{graphicx}
\usepackage{amsmath}
\usepackage{amsfonts}
\usepackage{amssymb}
\usepackage{amsthm}
\usepackage{hyperref}
\usepackage{xcolor}
\usepackage{todonotes}
\usepackage{overpic}

\def\ph0{\phi_0}

\newcommand{\blackboardbold}{\mathbb}

\newcommand{\BR}{\ensuremath{\blackboardbold{R}}}

\newcommand{\Z}{\ensuremath{\blackboardbold{Z}}}

\newcommand{\ovfrac}[2]{{#1}/{#2}}
\renewcommand{\ge}{\geqslant}
\renewcommand{\le}{\leqslant}
\newcommand{\abs}[1]{\lvert #1\rvert}

\newcommand{\ovdiff}[2]{\ovfrac{\mathrm{d}#1}{\mathrm{d}#2}}

\newtheorem{lemma}{Lemma}

\newtheorem{thm}[lemma]{Theorem}

\title{Uncountably many cases of Filippov's sewed focus}

\author{Paul Glendinning\thanks{School of Mathematics, University of Manchester, Manchester M13 9PL.}, S.~John Hogan\thanks{Corresponding author: \href{mailto:s.j.hogan@bristol.ac.uk}{s.j.hogan@bristol.ac.uk}},  Martin Homer,\\ Mike R. Jeffrey, Robert Szalai\thanks{Department of Engineering Mathematics, University of Bristol, Bristol BS8 1TW.}}

\begin{document}

\maketitle

\begin{abstract}
The sewed focus is one of the singularities of planar piecewise smooth dynamical systems. Defined by Filippov in his book {\it Differential Equations with Discontinuous Righthand Sides} (Kluwer, 1988), it consists of two invisible tangencies either side of the switching manifold. 
In the case of analytic focus-like behaviour, Filippov showed that the approach to the singularity is in infinite time. 

For the case of non-analytic focus-like behaviour, we show that the approach to the singularity can be in finite time. 
For the non-analytic sewed centre-focus, we show that there are uncountably many different topological types of local dynamics, including cases with  infinitely many stable periodic orbits, and show how to create systems with periodic orbits intersecting any bounded symmetric closed set.
\end{abstract}


\section{Introduction}
\label{section:intro}
Piecewise smooth systems have a wide range of applications in mechanics, engineering, control theory and biology \cite{dBBK, Jeffrey}. Many of the fundamental results on these systems are collected in Filippov's seminal book \cite{f88}.  Some of his results have been partly rediscovered (see \cite{Hog2015} for a case study from bifurcation theory). Other results are not well-known, or are themselves incomplete.

A general planar piecewise smooth differential equation  with switching manifold $\Sigma: y=0$ is given by
\begin{equation}\label{basicmodel}
\begin{array}{lll}
\dot x = P^+(x,y), &\dot y=Q^+(x,y), &\text{if $y>0$},\\
\dot x = P^-(x,y), &\dot y=Q^-(x,y), &\text{if $y<0$},
\end{array}\end{equation}
where the functions $P^\pm$ and $Q^\pm$ can be smoothly extended to a neighbourhood of $\Sigma$. We say that \eqref{basicmodel} is $C^k_*$ if $P^\pm, Q^\pm$ are $C^k$, $k\in \Z^+ \cup \{\infty\}$, and $C^\omega_*$ if they are analytic.\footnote{Note that Filippov \cite[p.~206]{f88} uses a more general definition: systems are said to be $C^k_*$ if $P^\pm$ and $Q^\pm$ are $C^k$ \textit{and} if $\Sigma$ is $C^{k+1}$, fixed and the same for all systems in the class.}

Filippov \cite[p.~220]{f88} describes a classification of singularities of \eqref{basicmodel} on $\Sigma$. Assuming that the singularity is at the origin $\text{O}:=(0,0)$, \textit{type 3 singularities} occur if
\begin{equation}\label{Type3}
P^\pm(0,0)\ne 0, \quad Q^+(0,0)=Q^-(0,0)=0.
\end{equation}
One type 3 singularity is the {\it sewed focus}\footnote{ Also called a \textit{fused}, \textit{merged} or \textit{stitched} focus.}, which is made up of two invisible tangencies either side of the switching manifold \cite[p.~234]{f88}. Locally, trajectories in $y>0$ are inverted parabolas and in $y<0$ they are parabolas, creating a flow that can be focus-like (see Figure~\ref{fig:Fil80}) or centre-like. The conditions for a sewed focus are, up to a change in the direction of time,
\begin{equation}\label{sf1}
P^+(0,0)=b^+>0, \quad P^-(0,0)=b^-<0,
\end{equation}
and
\begin{equation}\label{sf2}
xQ^\pm(x,0)<0\quad {\rm if}~x\ne 0
\end{equation}
for $x$ in some neighbourhood $V$ of $x=0$. The simplest sufficient condition to satisfy \eqref{sf2} locally is 
\begin{equation}\label{sf2strong}
Q^+_x(0,0)<0, \quad Q^-_x(0,0)<0.
\end{equation}
Note that \eqref{sf2} and \eqref{sf2strong} each imply that $\dot y$ has the same sign at $(x,0)$ for all $x\ne 0$ on both sides of $y=0$. Hence \eqref{basicmodel} admits unique solutions passing through $\Sigma$. There is no sliding\footnote{The vanishing of all derivatives with respect to $x$ up to order $k-1$ and such that
$\ovfrac{\partial^k Q^\pm}{\partial x^k}<0$
also defines a sewed focus.} and hence no Filippov vector field. 
\begin{figure}
\begin{center}
\begin{overpic}[angle=90, height=0.40\textheight]{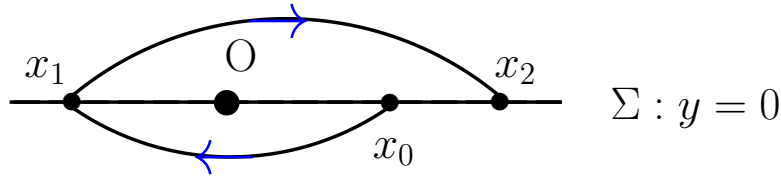}
\put(106,44.8){\Large{$\Sigma: y=0$}}
\put(37,44.7){\Huge{$\bullet$}}
\put(39.5,53){\Large{O}}
\put(33.7,35.3){{\color{blue}\Huge{$\leftarrow$}}}
\put(43,58.8){{\color{blue}\Huge{$\rightarrow$}}}
\put(5,52){\Large{$x_1$}}
\put(11.2,45.7){\Large{$\bullet$}}
\put(86,52){\Large{$x_2$}}
\put(85,45.7){\Large{$\bullet$}}
\put(65,38){\Large{$x_0$}}
\put(66,45.5){\Large{$\bullet$}}
\end{overpic}
 \end{center}
 \vspace{-3cm}
\caption{Flow around a sewed focus at the origin O, formed from two invisible tangencies, with switching manifold $\Sigma: y=0$, and $x_{0,2}>0$, $x_1<0$. Reproduction of \cite[Figure 80, p. 234]{f88}.}
\label{fig:Fil80}
\end{figure}

If \eqref{basicmodel} is $C^\omega_*$, there are just two cases \cite{f88} up to time reversal: the singularity is either i) a stable (unstable) focus, locally solutions tend to $(0,0)$ as $t\to \infty$ ($t\to-\infty$), and solutions reach $(0,0)$ in infinite time or ii) a centre and all solutions except the singularity are periodic locally. Although these cases are given in \cite{f88}, one aim of this paper is to clarify the result and simplify the proof (see Theorem~\ref{thm:analytic} below). 

Another aim is to gain an understanding of the non-analytic cases, which were not fully considered in \cite{f88}. In the case of focus-like behaviour we show that \emph{if the system is $C^k_*$, $k\in \Z^+ \cup \{\infty\}$, then the approach to the singularity can be in finite time} (see Theorem~\ref{thm:finitetime} below). 

We also consider more general sewed centre-focus\footnote{The centre-focus in smooth systems is well-defined \cite[p. 139]{Perko}.} behaviour, which was not considered at all\footnote{``[T]he cases of [sewed] centre-focus are not considered here'' \cite[p.~222]{f88}.} in \cite{f88}. In particular we show that there are uncountably many different local structures\footnote{In fact, Filippov appears to have been aware of the complexity possible in this case, stating that ``...there are infinitely many [sewed] [centre-foci]'' \cite[p.~250]{f88}. But he offers no proof of this statement. The absence of any proof in the literature led us to results that follow.} near this singularity (see Theorem~\ref{thm:remarkable} below).  

Our paper is organized as follows. In section~\ref{section:classic} we review the theoretical framework and results attributed by Filippov \cite[p.~234]{f88} to Bautin and Leontovich \cite{BL76} and Skriabin \cite{Skriabin}, and give the description of the local structure for piecewise analytic systems, $C^\omega_*$ (Theorem~\ref{thm:analytic}). This theorem has been outlined before \cite[§19.4]{f88}. It is included here because it does not appear to be well known. Our proof clarifies and simplifies the one given in \cite[§19.4]{f88}. 

In section~\ref{section:finitetime} we show that the approach to the singularity can be in finite time for piecewise non-analytic systems, $C^k_*$, $k\in \Z^+ \cup \{\infty\}$. In section~\ref{section:general} we describe an example of a sewed focus with infinitely many isolated stable periodic orbits (Theorem~\ref{thm:gensewedfocus}) and show how this can be generalized to create systems with periodic orbits intersecting any bounded symmetric closed set (Theorem~\ref{thm:remarkable}).



\section{The analytic case}
\label{section:classic}
If \eqref{basicmodel} is $C^\omega_*$, there are two types of local behaviour, up to time reversal.

\begin{thm}\label{thm:analytic}Suppose that the origin is a sewed focus of a $C^\omega_*$ system \eqref{basicmodel}. Then the origin is either i) a centre or ii) all solutions approach the singularity in infinite time as either $t\to \infty$ (stable sewed focus) or $t\to -\infty$ (unstable sewed focus). 
\end{thm} 

\begin{proof} 
The approach of Bautin and Leontovich \cite{BL76}, as described by Filippov \cite[p.~234]{f88}, is to define 
functions $\sigma^\pm$ in a neighbourhood of the singularity. 

With the time convention as in
\eqref{sf1}, in $y\ge 0$, the function $\sigma^+$ takes $(x_1,0)$ on $\Sigma$, where $x_1<0$, to the next intersection of the flow with $\Sigma$, given by $(x_2,0)$, in forwards time, where $x_2 >0$. Hence in $x<0$, $\sigma^+(x_1)=x_2$, locally; see Figure~\ref{fig:Fil80}.
We can extend $\sigma^+$ to $x\ge 0$ by defining $\sigma^+(0)=0$ and $\sigma^+(x_2)=x_1$ with $x_2>0$ iff $\sigma^+(x_1 )=x_2$. 

The function $\sigma^-$ on $\Sigma$ following the flow in $y\le 0$, is defined in the same way: in $x<0$, $\sigma^-(x_1)=x_0>0$ and in $x\ge0$,  $\sigma^-(0)=0$ and $\sigma^-(x_0)=x_1$.

Hence
\begin{equation}\label{ssid}
\sigma^\pm (\sigma^\pm (x))=x.
\end{equation} 

We claim that $\sigma^\pm$ extend to analytic functions on a neighbourhood of $x=0$ with
\begin{equation}\label{sig0}
\sigma^\pm (0)=0,\quad \sigma^{\pm \prime}(0)=-1.
\end{equation}
It is straightforward to show that $\sigma^{\pm \prime}(0)=-1$. Differentiate both sides of \eqref{ssid} with respect to $x$, to get $\sigma^{\pm \prime}(\sigma^{\pm }(x))\sigma^{\pm \prime}(x)=1$. By assumption $\sigma^\pm (0)=0$. Hence $(\sigma^{\pm \prime} (0))^2=1$, $\sigma^{\pm \prime} (0)=\pm1$. Since $x\sigma^{\pm}(x)<0$, then $\sigma^{\pm \prime}(0)=-1$.

We delay the justification of the claim to analyticity to the end of the proof.
Let 
\begin{equation}\label{chidef}
\chi (x)\equiv\sigma^+(x)-\sigma^-(x)
\end{equation}
which is analytic on a neighbourhood of the origin.
Using \eqref{sig0}, we have
\begin{equation}\label{chideriv}
\chi (0)=0, \quad \chi^\prime (0)=0.
\end{equation}

We begin by showing that if $\chi (x)>0$ ($<0$) $\forall x<0$ in a neighbourhood of the singularity, then  $\chi (x)>0$ ($<0$) $\forall x>0$ (locally) as well. 

Clearly $\sigma^{\pm\prime}(x)<0$ on a neighbourhood of the origin so $\sigma^\pm$ are decreasing functions\footnote{Also true since orbits of a
vector field on the plane can not intersect.} of $x$ locally. Suppose that $\chi (x)>0$, that is $\sigma^-(x)<\sigma^+(x)$, for all $x<0$ in a neighbourhood of $x=0$. If we let $\sigma^+ (x)=x_+, \sigma^- (x)=x_-$, then $x_+>x_->0$. Since $\sigma^+$ is decreasing this implies that $\sigma^+(x_+)<\sigma^+(x_-)$. Subtract $\sigma^- (x_-)$ from both sides and hence
\[
0 = \sigma^+(x_+)-\sigma^- (x_-) < \sigma^+ (x_-)-\sigma^-(x_-)
\]
where the equality holds because $\sigma^+ (x_+)=\sigma^- (x_-)=\sigma^\pm (\sigma^\pm (x))=x$ using \eqref{ssid}. Hence if $\chi (x)>0$ in $x<0$ locally then $\chi (x)>0$ in $x>0$ locally. This follows from Figure~\ref{fig:Fil82}, since the functions $z=\sigma^{\pm}(x)$ both decrease with $x$ and are symmetric about the line $z=x$. If $\chi (x)<0$ then the argument is completely analogous.

\begin{figure}
\begin{center}
\begin{overpic}[height=0.35\textheight]{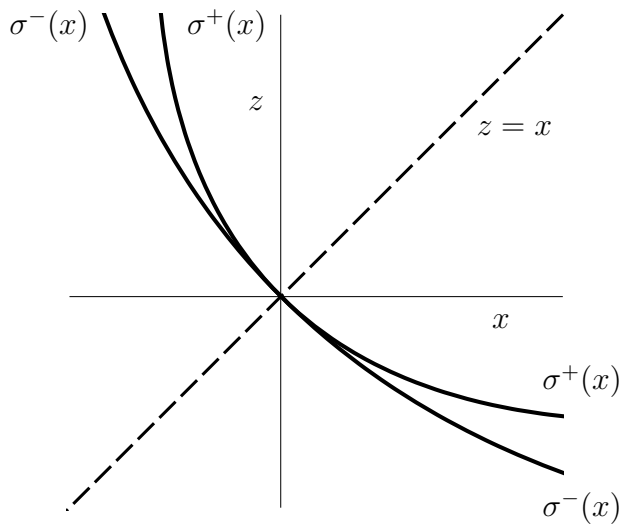}
\put(95,25){$\sigma^+(x)$}
\put(95,0){$\sigma^-(x)$}
\put(25,95){$\sigma^+(x)$}
\put(-10,95){$\sigma^-(x)$}
\put(37,80){$z$}
\put(85,37){$x$}
\put(82,75){$z=x$}
\end{overpic}
 \end{center}
\caption{Demonstration that if $\chi (x)\equiv\sigma^+(x)-\sigma^-(x)>0$ $\forall x<0$ in a neighbourhood of the singularity, then  $\chi (x)>0$ $\forall x>0$ (locally) as well. 
Reproduction of \cite[Figure 82, p. 235]{f88}.}
\label{fig:Fil82}
\end{figure}

Since $\chi$, $\sigma^+$ and $\sigma^-$ are analytic then either i) $\chi (x)\equiv 0$ on a neighbourhood of $x=0$, which implies that all local solutions are periodic and so the origin is a centre, or ii) there exists $r\ge 2$ such that $\chi^{(k)}(0)=0$, $k=1, \dots ,r-1$ and 
$\chi^{(r)}(0)\ne 0$. In this case the previous argument implies that $r$ is even, $r=2p$ say, and so locally
\begin{equation}\label{chiexp}
\chi (x)=\frac{\chi^{(2p)}(0)}{(2p)!}x^{2p}+o(x^{2p}).
\end{equation}
Hence in a neighbourhood of $x=0$ either $\chi (x) <0$ for all $x\ne 0$ (a stable focus) or $\chi (x) >0$ for all $x\ne 0$ (an unstable focus).

The proof that the approach of the origin takes place in infinite time for the sewed focus is given in \cite[pp.~237--238]{f88}. Suppose that the first non-zero derivative of $\chi$ at the origin is at order $2p$, as in \eqref{chiexp}. Let successive intersections of a trajectory with the $x$-axis be denoted by $x_{n}$. By definition \eqref{chidef}
\[
x_{2n+2}=x_{2n}+\chi (x_{2n+1}),
\]
where $x_{2n}>0, \, \forall \, n$. Then \eqref{sig0} implies that $x_{2n+1}/x_{2n}\to -1$ as $n\to \infty$.  Hence by \eqref{chiexp}
\[
x_{2n+2}-x_{2n}=ax^{2p}_{2n+1} + o(x^{2p}_{2n+1})
\]
where $a=\ovfrac{\chi^{(2p)}(0)}{(2p)!}\ne 0$. Consider the stable\footnote{The unstable sewed focus $a>0$ is similar after reversing the direction of time.} sewed focus ($a<0$). Then for sufficiently large $n$  there exists $A >0$ such that  
\[
x_{2n+2}\ge x_{2n}-Ax_{2n}^{2p}\ge  x_{2n}-Ax_{2n}^2.
\]
We will now show that $\sum_1^\infty x_{2n}$ diverges.  Since the speed in the $x$-direction is bounded away from zero in a neighbourhood of the origin, then the time taken to approach the origin will also then diverge.

Let $u_{n+1}=u_n-Au_n^2$. If $u_1=x_2$ then $\sum_1^\infty x_{2n}\ge \sum_1^\infty u_n$. Let $v_n=\frac{x_2}{2n}$ for $n\ge 2$, chosen because $\sum _1^\infty v_n$ diverges.

Now compare $u_n$ with $v_n$. Clearly $u_1> v_1$. Now suppose that $u_n> v_n \, \forall \, n \in [2,m]$. Then
\[
u_{m+1}=u_m-Au_m^2 
>\frac{x_2}{2m}-A\frac{x_2^2}{4m^2}. 
\]
So $u_{m+1}>v_{m+1}$ provided
\[
\frac{x_2}{2m}-A\frac{x_2^2}{4m^2}>\frac{x_2}{2(m+1)}
\]
or, on rearranging
\[2m^2x_2>Ax_2^2m(m+1).\]

Thus if $x_2$ is sufficiently small then $u_m>v_m$. Hence  both $\sum_1^\infty x_{2n}$ and $\sum_1^\infty u_n$ diverge, by the comparison test.
 
It remains only to demonstrate that $\sigma^{\pm}$ are analytic on a neighbourhood of $x=0$.  Filippov \cite{f88} refers to classic analysis of Andronov et al.\ \cite{ALGM}, but here we sketch a rather simpler explanation due to Chillingworth \cite{DC2016}.

Consider $y\ge 0$ (the argument in $y<0$ is similar). We have, from \eqref{basicmodel},
\[
\dot x = P^+(x,y), \quad \dot y =Q^+(x,y)
\]
with, from \eqref{sf1} and \eqref{sf2},  
\[
P^+(0,0)=b^+>0, \quad xQ^+(x,0)<0.
\]
Both $P^+$ and $Q^+$ are analytic in a neighbourhood of $(0,0)$. Then write $P^+(x,y)=b^++p^+(x,y)$ with $p^+$ analytic and $p^+(0,0)=0$. Write 
$Q^+(x,y)=-x^{2k+1}r^+(x)+ys^+(x,y)$ for some $k\ge 0$, with $r^+(0)=\frac{1}{(2k+1)!}\frac{\partial^{2k+1}Q^+}{\partial x^{2k+1}}(0,0)\ne 0$ (an odd power is needed to ensure that $xQ^+(x,0)<0$ locally), with both $r^+$ and $s^+$ analytic near $(0,0)$.

\begin{figure}
\centering
\includegraphics[width=10cm]{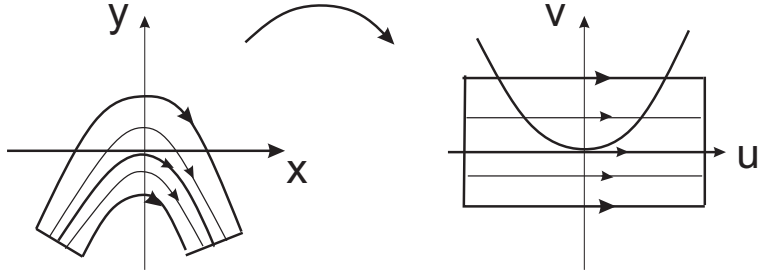}
\caption{Change of coordinates takes $y=0$ to a generalized parabola.}
\label{fig:flowbox}
\end{figure}

Then $\ovdiff{y}{x}=F^+(x,y)$ where $F^+$ is locally analytic and given by
\begin{equation}\label{F} 
F^+(x,y)=\frac{Q^+(x,y)}{P^+(x,y)}=-x^{2k+1}R^+(x)+yS^+(x,y)
\end{equation}
with $R^+$ and $S^+$ analytic, close to $r^+$ and $s^+$ up to scaling. 
Thus solutions must lie\footnote{Substitute into \eqref{F}.} on analytic curves of the form
\[
y=c-x^{2k+2}W^+(x)
\]
where $c$ is constant, $W^+$ is analytic, and $W^+(0)>0$. 

Consider a flow box in a neighbourhood of the origin bounded by two trajectories, one with $c>0$ and the other with $c<0$. By an analytic change of coordinates $(x,y)\to (u,v)$ this can be transformed\footnote{One example of a simple transformation is $(x,y)\to (u,v)=(x,y+x^{2k+2}W^+(x))$. In that case, $y=0$ would map to $v=W^+(u)u^{2k+2}$ and so $C(1+G(u))=W^+(u)$.} to a flow where solutions lie on horizontal lines as shown in Figure~\ref{fig:flowbox}. Then the line $y=0$ maps to a curve
\[
v=C(1+G(u))u^{2k+2}.
\]
An analytic near identity change of coordinates, leaving $u$ unchanged, of the form $(u,v) \mapsto (u,w)$,  brings this to 
\[
w=u^{2k+2}.
\]
In these new coordinates the return map is simply $u\to -u$ which is analytic. Hence $\sigma^+$ is analytic, concluding the proof.

\end{proof}

Provided the smallest non-zero derivative of $\chi$ exists, the proof also works for $C^k$ functions.

\begin{thm}\label{thm:strong}If the origin is a sewed focus of a $C^k_*$ system, $k\ge 2$, and  for some $p$ such that
$2p\le k$ the function $\chi$ is $C^{2p}$ with
\[
\chi^{(s)}(0)=0, ~ s=1,2,\dots, 2p-1, \quad \chi^{(2p)}(0)\ne 0,
\]
then locally trajectories approach the origin in infinite time, as $t\to \infty$ if $\chi^{(2p)} (0)<0$ and as $t\to-\infty$ if $\chi^{(2p)} (0)>0$.
\end{thm}

\section{Finite time approach}
\label{section:finitetime}
The approach to the origin is asymptotic for $C^k_*$ systems satisfying the conditions of Theorem~\ref{thm:strong}. The aim of this section is to show that \textit{finite} time approach is also possible for these non-analytic systems.

\begin{thm}\label{thm:finitetime}For all $k\in \Z^+ \cup \{\infty\}$, there exists a $C^k_*$ system of the sewed focus and a neighbourhood $U$ of the origin such that if $(x_0,y_0)\in U$ then the trajectory through $(x_0,y_0)$ tends to $(0,0)$ in finite time.
\end{thm}

\begin{proof}
First we assume $k\in \Z^+$. Let $H(x)$ denote the Heaviside function,
\begin{equation}\label{eq:Heav}
H(x)=\begin{cases}1 & {\rm if}~x\ge 0\\ 0 & {\rm if}~x<0 .\end{cases}
\end{equation}
Consider system \eqref{basicmodel} with
\begin{equation}\label{sewedfocc1}
\begin{aligned}
P^+(x,y)&=1 & Q^+(x,y)&=-4x^3H(x)-8x^7H(-x), \\
P^-(x,y)&=-1, & Q^-(x,y)&=-4x^3H(-x)-8x^7H(x).
\end{aligned}
\end{equation}
The symmetry and the values of the constants (other than their signs) are not important, and have been chosen for algebraic convenience. This example is $C^2_*$ and has a stable sewed focus at the origin. There is no sliding. 

Given an initial point $(-x_0,0)$ with $x_0>0$ sufficiently small, the solution will strike $\Sigma$ at points $(x_1,0), (-x_2,0), \dots$ with $x_i>0$. The total time $T$ taken is given by
\[
T=x_0+2\sum_{r=1}^\infty x_r,
\]
since from \eqref{sewedfocc1} the speed in the $x$-direction is equal to 1.

Thus if $T<\infty$ the singularity is contained in the approaching trajectory, and is reached in finite time. The trajectories of systems in the same topological class are mapped to each other. So the inclusion of the origin O in an approaching trajectory implies the existence of a new topological class, not possible in the analytic case.

If $y>0$ then integral curves are given by
\begin{equation}\label{focintcur}
y=c^--x^8, \ \ x<0, \quad y=c^+-x^4, \ \ x>0
\end{equation}
where the constants $c^{\pm}$ are determined by initial conditions. The integral curves in $y<0$ are determined by symmetry.

Thus a solution starting at $(-x_0,0)$, $x_0>0$, lies on the integral curve
\[ y=x_0^8-x^8 \]
and strikes the $y$-axis after time $x_0$ at $y=x_0^8$. It now continues on   an integral curve of the form $y=c^+-x^4$. Since $y=x_0^8$ when $x=0$, $c^+=x_0^8$. This integral curve in turn strikes $\Sigma$ at $(x_1,0)$, $x_1>0$, where $x_0^8-x_1^4=0$, i.e. $x_1=x_0^2$, after a further $x_1$ units of time.

Now in $y<0$, which by symmetry is effectively equivalent, the solution through $( x_1,0)$ strikes $\Sigma$ at $(-x_2,0)$ where $x_2=x_1^2=x_0^4$ after time $x_1+x_2$. By induction, the infinite sequence of intersections is given by the points $((-1)^{(r+1)} x_r,0)$ where
\[
x_r=x_0^{2^r}
\]
and hence
\begin{equation}
T=x_0+2\sum_{r=1}^\infty x_0^{2^r}
\label{totaltime}
\end{equation}
The sum \eqref{totaltime} is certainly less than twice the sum of all powers of $x_0$, which converges if $x_0<1$ as it is a geometric progression, so this sum also converges to a finite value. 

Replacing $4x^3$ and $8x^7$ in \eqref{sewedfocc1} by $2kx^{2k-1}$ and $4kx^{4k-1}$, respectively, with $k\ge 3$, the relation between subsequent intersections is unchanged and so this example can be made $C^r_*$ for any $r<\infty$.

We now give an example\footnote{We are extremely grateful to an anonymous reviewer for this example.} of a $C^\infty_*$ system with finite time approach to the origin. 
Consider system \eqref{basicmodel} with
\begin{align}
        P^+(x,y)&=1, & Q^+(x,y)&=
        \begin{cases}\label{reviewerexample}
            -\frac{2}{x^3}e^{-1/x^2}, & x<0,\\
            0, & x=0,\\
            -\frac{1}{x^2}e^{-1/x}, & x<0
        \end{cases}\\
        P^-(x,y)&=-1, & Q^-(x,y)&=
        \begin{cases}\nonumber
            \frac{1}{x^2}e^{-1/x}, & x<0\\
            0, & x=0,\\
            -\frac{2}{x^3}e^{-1/x^2}, & x>0.
        \end{cases}
    \end{align}
Then $x_r = x^2_{r-1}$ for all $r \ge 1$ and $T$, also given by equation \eqref{totaltime}, is finite.

\end{proof}

Hence there is a new topological class with finite time approach to the singularity for piecewise non-analytic systems, $C^k_*$ with $k\in \Z^+ \cup \{\infty\}$.

\section{The general sewed focus}
\label{section:general}
In this section, we will prove that there are three possibilities for the local dynamics near a sewed focus, independent of the time taken to approach the singularity\footnote{The issue of finite versus infinite time approach in case \textit{(a)} is not included here as it was considered in section~\ref{section:finitetime}.}. Also we show how to construct a non-analytic piecewise smooth system with infinitely many isolated stable periodic orbits.

\begin{thm}\label{thm:gensewedfocus}
Consider an isolated sewed focus singularity of a  $C^k_*$ system \eqref{basicmodel}, where either $1\le k \le \infty$ or $k=\omega$. Let $\chi$ be defined by \eqref{chidef}. Then the singularity is either 
\begin{enumerate}
\item[(a)] a stable (or an unstable) sewed focus (with $\chi (x)\ne 0$ for all $x\ne 0$ locally) or
\item[(b)] a sewed centre (with $\chi (x) \equiv 0$ locally) or
\item[(c)] a sewed centre-focus (where $\exists~x\ne 0$ arbitrarily close to $x=0$ such that $\chi (x)=0$).
\end{enumerate}
If $k=\omega$ then only cases (a) or (b) can occur, otherwise all three cases can occur.
\end{thm}

\begin{proof}
For $C^\omega_*$ systems, by the results of section~\ref{section:classic}, either i) $\chi (x)\equiv 0$ (a sewed centre) or ii) $\exists~k\ge 1$ such that $\chi^{(2k)}(0)\ne 0$ and all lower derivatives are zero. In the latter case, either $\chi (x)>0$ or $\chi (x)<0$ for all $x\in B(0,\delta )\backslash\{0\}$ for $\delta >0$ small enough, giving either an unstable or a stable sewed focus respectively. Nothing else is possible.


For $C^k_*$ systems, $1\le k\le \infty$ we  consider vector fields in $y<0$ of the form
\begin{equation}\label{infbelow}
 P^-(x,y)=-1, \quad Q^-(x,y)=-2x.
\end{equation}
In $y<0$ solutions  lie on parabolas of the form $y=-c^2+x^2$. In $y>0$ define
\begin{equation}\label{infabove}
 P^+(x,y)=1, \quad Q^+(x,y)=-2x+2kx^{2k-1}f(\textstyle{\frac{1}{x}})-x^{2k-2}f^\prime (\textstyle{\frac{1}{x}})
\end{equation}
with $f(x)$ odd. By design, $Q^+$ is $C^k_*$ for sufficiently differentiable $f$. So \eqref{basicmodel}, with \eqref{infbelow} and \eqref{infabove} is $C^k_*$ and satisfies \eqref{Type3}.

The flow in $y>0$ has been chosen so that solutions through $(-x_0,0)$  lie on curves
\begin{equation}\label{intcurves}
y(x)=-x^2 +x^{2k}f(\textstyle{\frac{1}{x}})+x_0^2+x_0^{2k}f(\textstyle{\frac{1}{x_0}})
\end{equation}
where we use the fact that $f$ is odd. Hence
\[
y(x_0)=2 x_0^{2k}f(\textstyle{\frac{1}{x_0}}).
\]
The zeros of this function correspond precisely to the closed orbits of the vector field,
provided that $y(x)>0$ for all $x\in(-x_0,x_0)$.
Thus if we choose
\[
f(x)=\sin 2\pi x
\]
and $k\ge2$, then $\forall \, n\in\BR$, $n\ge4$, we have  
\[
x_0=\frac{1}{n},
\]
which lies on a (different) periodic solution for each $n$; those with $n$ even are asymptotically stable and those with $n$ odd are unstable.

Choosing different forms of $f$ gives different zero sets. Hence case $(c)$ comes in infinitely many different types, depending on whether the components of sets with $\chi(x)=0$ are intervals or points, and how these sets are connected.
\end{proof}

We can now make some generalizations. We begin by recalling the following result.

\begin{lemma}\label{fc}Let $E$ be any symmetric compact subset of $\BR$ not containing the origin (i.e. $E=-E$ and $0\notin E$), and $k\in \Z^+\cup\{\infty\}$. Then there exists an odd $C^{k+1}$ function $f_E$ which, if $k<\infty$,  is not a $C^{k+2}$ function such that $f_E(x)=0$ if and only if $x\in E\cup\{0\}$.
\end{lemma}

\begin{proof}
Given a symmetric compact set $E$ and $0\notin E$, let $E_0=E\cup\{0\}$ and $d(x,E_0)=\min_{y\in E_0} \abs{x-y}$. Define
\begin{equation}\label{symfun}
g_E(x)=-\textrm{sign}(x)[d(x,E_0)]^{k+2}.
\end{equation}
Clearly $g_E$ is a continuous, odd function. Moreover, $g_E(x)=0$ iff $x\in E_0$. Since $E_0$ is closed its complement is open and hence is a countable union of disjoint open intervals. The function $g_E$ is $C^{k+1}$ except at the mid-points of each of these intervals. But $g_E$ can be modified in a neighbourhood of each of these mid-points so that the modification $f_E$ is $C^{k+1}$ there and still odd. This proves the first part of the result if $k<\infty$. 

If $k=\infty$ it is still the case  that the complement of $E_0$ is open and is thus a  union of disjoint open intervals. This union is countable (as each open interval contains a rational) and so we can label the intervals $J_i=(a_i-\delta_i,a_i+\delta_i)$, so that the centre points are $a_i$ and the length of $J_i$ is $2\delta_i$, $i=1,2, \dots$. Now let $\psi (x)$ denote the standard $C^\infty$ bump function on the unit interval:
\begin{equation}\label{bump}
\psi(x)=\begin{cases}\exp\left(-\frac{1}{1-x^2}\right)& \text{if $\abs{x}<1$}\\ 0 & \text{if $\abs{x}\ge 1$},
\end{cases}\end{equation}
 and define 
\begin{equation}\label{cinfinity}
f_E(x)=\sum_{i=1}^\infty -\textrm{sign}(a_i)\delta_i\psi \left(\frac{x-a_i}{\delta_i}\right).
\end{equation}
Then $f_E$ is $C^\infty$ and zero if and only if $x\in E_0$ ($0\in E_0$ by definition so $a_i\ne 0$). Finally, $f_E$ is an odd function, since $E_0$ is symmetric. 

\end{proof}

The separation of $x=0$ from $E$ in the statement of Lemma~\ref{fc} is because the interpretation of the points in $E$ for the sewed focus (they are on periodic orbits) is different from the sewed focus itself which has $x=0$.  

Hence we have the rather remarkable theorem.

\begin{thm}\label{thm:remarkable}Let $E$ be any symmetric compact subset of $\BR$ not containing the origin. Then there exists a $C^k_*$ Filippov system \eqref{basicmodel}, $1\le k\le \infty$, with an isolated type 3 singularity, for which $(x_0,0)$ lies on a periodic orbit if and only if $x_0\in E$.
\end{thm}

\emph{Proof:} Choose $1\le k\le\infty$ and let $f_E(x)$ be the odd $C^{k+1}$ or $C^\infty$ function of Lemma~\ref{fc}. Define the $C^k_*$  Filippov system using \eqref{infbelow} in $y<0$ and 
\begin{equation}\label{abovenew}
 P^+(x,y)=1, \quad Q^+(x,y)=-2x+x^2f_E^\prime (x)+2x f_E(x)
\end{equation}
in $y>0$, which satisfies \eqref{Type3} since $f_E^\prime$ is continuous.   Furthermore, as $f_E^{\prime\prime}$ is continuous and $f_E(0)=0$, $Q^{\pm}_x(x,y)<0$ in a neighbourhood of the origin. Hence in this neighbourhood $Q^\pm(x,y)>0$ for $x<0$ and $Q^\pm(x,y)<0$ for $x>0$, which guarantees that there is no sliding in the neighbourhood of the origin. Thus solutions in $y>0$ through $(-x_0,0)$ lie on sets
\begin{equation}\label{zeroeq}
y(x)=-x^2+x^2f_E(x)+x_0^2+x_0^2f_E(x_0)
\end{equation}
by integrating $\ovdiff{y}{x}=\ovfrac{Q^+}{P^+}$, using \eqref{abovenew}.

The flow in $y<0$ implies that an orbit through $(x_0,0)$, $x_0>0$, next intersects the switching manifold at $(-x_0,0)$ and hence  a periodic orbit through $(\pm x_0,0)$ exists iff the flow in $y>0$ takes $(-x_0,0)$ to $(x_0,0)$, i.e.\ that $\sigma^+(-x_0)=x_0$.

First suppose that $x_0>0$ and $(x_0,0)$ lies on a periodic orbit. Then following the flow in $y<0$, $(-x_0,0)$ lies on the same periodic orbit, and hence $y(\pm x_0)=0$. Then  \eqref{zeroeq} gives
\[
y(x_0)=2x_0^2f_E(x_0)=0,
\]
and so $x_0\in E$.

To prove sufficiency, suppose that $x_0>0$ and $x_0\in E$. Since $E$ is symmetric $-x_0\in E$ too, and so $f_E(\pm x_0)=0$. Then, recalling that $f_E$ is odd, \eqref{zeroeq} gives
\[
y(x_0)=2x_0^2f_E(x_0)=0,
\]
so an orbit through $(-x_0,0)$ crosses the switching manifold at $(x_0,0)$. Suppose this is not the first crossing, i.e.\ that there exists at least one root  $y(x^\star)=0$ with $x^\star\in(-x_0,x_0)$.  For all $x\in E_0$, $\abs{x}<x_0$,  
\begin{equation}
  y(x)=-x^2+x_0^2+x^2f_E(x)=-x^2+x_0^2>0,
\label{zeroeqinE}
\end{equation}
so $x^\star$ must be in the complement of $E_0$, in one of the intervals $J_i$, see Lemma~\ref{fc}.  Equation \eqref{zeroeqinE} also implies that $y(x)>0$ at the endpoints of all these intervals; since $y(x)$ is continuous, and has no tangencies (as $Q^\pm(x,y)\neq0$ for $x\neq0$), there must be an even number of roots of $y(x)$ in the interval $J_i$ containing $y^\star$.  But this contradicts the fact that $Q^\pm(x,y)$ do not change sign on any interval $J_i$, as noted above. So there is no root $x^\star\in(-x_0,x_0)$, and the trajectory in $y>0$ through $(-x_0,0)$ next crosses the switching manifold at $(x_0,0)$, and hence is a periodic orbit for any $x_0\in E$.

Finally, let $x_0>0$ with $x_0\notin E$. Again, by symmetry of $E$, $-x_0\notin E$ and this time \eqref{zeroeq} gives
\[
y(x_0)=2x_0^2f_E(x_0)\neq0,
\]
so the trajectory through $(-x_0,0)$ cannot be a periodic orbit.

%

\rightline\qed 

In Theorem~\ref{thm:gensewedfocus}, we have shown that there are uncountably many topological classes of the sewed centre-focus, a type of non-analytic isolated sewed focus singularity of a  $C^k_*$ system \eqref{basicmodel}, where $1\le k \le \infty$. This result was generalized in Theorem~\ref{thm:remarkable}.

\section{Conclusion}\label{sec:conclusion}
In this paper we have analyzed the sewed focus, a singularity of planar piecewise smooth systems. Although some analysis of this case is presented by Filippov \cite{f88}, we believe that many of our results are new. 
In the case of non-analytic focus-like behaviour we show that the approach to the singularity can be in finite time, forwards or backwards as appropriate. 
More generally, for the non-analytic sewed centre-focus, we show that there are uncountably many different topological types of local dynamics, including cases with  infinitely many stable periodic orbits.

The proof of our results uses a simple construction. By considering particular cases in greater detail it is then possible to define $C^k_*$ systems, $k\in \Z^+ \cup \{\infty\}$, which have \emph{infinitely many stable periodic orbits}. This is another case of `infinitely many sinks' arising in piecewise smooth systems, although the mechanism is totally different from the examples for the border collision normal form \cite{SimpsonInf}. See also recent work \cite{Huzak2022}, which shows that the number of limit cycles for regularized piecewise polynomial systems is unbounded.  

Hence Hilbert's 16$^{th}$ Problem, concerning the number of stable periodic orbits of planar differential equations, is not relevant for the $C^k_*$ sewed focus. This problem has attracted considerable recent attention in the case of piecewise {\it linear} systems with a linear switching manifold \cite{Spain, Spain2}. In contrast, our paper answers the question for more general piecewise smooth systems.

The sewed focus is one of the earliest examples of type 3 singularities to be considered in the applied literature, with examples from the 1930s being included in Andronov et al.\ \cite[p.~511]{AKV}. These examples are $C^\omega_*$, so our more interesting results described here do not apply. But the effects described may arise in examples with more complicated switches which change higher derivatives of the flow in either $y>0$ or $y<0$ or both, as shown in section~\ref{section:finitetime}.

\medskip
\emph{Acknowledgements:} We are grateful to David Chillingworth (University of Southampton) for his insights on the analyticity of return maps, and to James Montaldi and Jonathan Fraser (University of Manchester) for helpful discussions, especially concerning Lemma~\ref{fc}. We are extremely grateful to an anonymous reviewer who supplied us with example \eqref{reviewerexample}, to complete the proof of Theorem~\ref{thm:finitetime}. SJH would like to thank the Hungarian Academy of Sciences for the award of a Distinguished Guest Scientist Fellowship, during which part of this paper was written.

\end{document}